\documentclass[fleq]{amsart}
\usepackage{amssymb,amsmath,latexsym,amsbsy,color,enumerate}
\numberwithin{equation}{section}

\newtheorem{theorem}{Theorem}
\newtheorem{definition}[theorem]{Definition}
\newtheorem{lemma}{Lemma}

\newtheorem{corollary}{Corollary}

\begin{document}
 \title{On automorphisms of blowups of $\mathbb{P}^3$}
 \author{Tuyen Trung Truong}
    \address{Indiana University, Bloomington IN 47405}
 \email{truongt@indiana.edu}
\thanks{}
    \date{\today}
    \begin{abstract}Let $\pi :X\rightarrow \mathbb{P}^3$ be a finite composition of blowups along smooth centers. We show that for "almost all" of such $X$, if $f\in Aut(X)$ then its first and second dynamical degrees are the same. We also construct many examples of finite blowups $X\rightarrow \mathbb{P}^3$, whose automorphism group  $Aut(X)$ has only finitely many connected components.
    
We also present a heuristic argument showing that for a "generic" compact K\"ahler manifold $X$ of dimension $\geq 3$, the automorphism group $Aut(X)$ has only finitely many connected components. 
    
\end{abstract}
\maketitle
\section{Introduction}
While there are many examples of compact complex surfaces having automorphisms of positive entropies (works of Cantat \cite{cantat}, Bedford-Kim \cite{bedford-kim2}\cite{bedford-kim1}\cite{bedford-kim}, McMullen \cite{mcmullen3}\cite{mcmullen2}\cite{mcmullen1}\cite{mcmullen}, Oguiso \cite{oguiso3}\cite{oguiso2}, Cantat-Dolgachev \cite{cantat-dolgachev}, Zhang \cite{zhang}, Diller \cite{diller}, D\'eserti-Grivaux \cite{deserti-grivaux}, Reschke \cite{reschke},...), there are few interesting examples of manifolds of higher dimensions having automorphisms of positive entropies (Oguiso \cite{oguiso1}\cite{oguiso}, Oguiso-Perroni \cite{oguiso-perroni},...). Some restrictions on projective $3$-manifolds having automorphisms of positive entropies are known (Zhang \cite{zhang2}\cite{zhang1},...).  On blowups of $\mathbb{P}^3$ or of products of projective spaces $\mathbb{P}^k$, only pseudo-automorphims of positive entropies are constructed up to date (Bedford-Kim \cite{bedford-kim3}, Perroni-Zhang \cite{perroni-zhang}, Blanc \cite{blanc},...). 

This paper concerns the dynamical degrees and topological entropies of automorphisms of finite blowups $X\rightarrow \mathbb{P}^3$. (Definitions of dynamical degrees and topological entropies are given in the next section.)  

\begin{theorem}
Let $X=X_n\rightarrow X_{n-1}\rightarrow \ldots \rightarrow X_1\rightarrow X_0=\mathbb{P}^3$ be a finite composition of blowups along smooth centers. Assume that each $X_{j+1}\rightarrow X_j$ ($0\leq j\leq n-1$) is either

1) A blowup of $X_j$ at a point 

or

2) A blowup of $X_j$ along a smooth curve $C\subset X_j$, so that $c_1(X_j).C\not= 2(g-1)$, where $c_1(X_j)$ is the first Chern class of $X_j$ and $g$ is the genus of $C$.

Then, for every $f\in Aut(X)$, its dynamical degrees satisfy $\lambda _1(f)=\lambda _2(f)$.  

\label{TheoremDynamicalDegreesAutomorphismsBlowupP3}\end{theorem}

If we incorporate the constructions from Theorem \ref{TheoremAutomorphismBlowupP3} below into Theorem \ref{TheoremDynamicalDegreesAutomorphismsBlowupP3}, we see that for "almost all" finite blowups $X\rightarrow \mathbb{P}^3$, if $f\in Aut(X)$ then its first and second dynamical degrees are the same. Next we construct many examples of finite blowups $X\rightarrow \mathbb{P}^3$ so that every element of $Aut(X)$ has zero topological entropy. In Corollary \ref{CorollaryAutomorphismGroupBlowupP3}, we show that most of the examples constructed in Theorem \ref{TheoremAutomorphismBlowupP3} satisfies the stronger constraint that their automorphism group $Aut(X)$ have only finitely many connected components.  

\begin{theorem}  
Let $X=X_n\rightarrow X_{n-1}\rightarrow \ldots \rightarrow X_1\rightarrow X_0=\mathbb{P}^3$ be a finite composition of blowups along smooth centers. Assume that $X_1\rightarrow X_0=\mathbb{P}^3$ is a blowup of a finite number of points $p_1,\ldots ,p_s\in \mathbb{P}^3$ and smooth curves $C_1,\ldots ,C_t\subset \mathbb{P}^3$ which are in general positions, that is no point belongs to a curve and two distinct curves are disjoint. 

Moreover, assume that each $X_{j+1}\rightarrow X_j$ ($1\leq j\leq n-1$) is either

1) A finite composition of blowups, the images in $X_j$ of the exceptional divisors are points;

or 

2) A blowup of $X_j$ along a smooth curve $C\subset X_j$ so that $\gamma =c_1(X_j).C+2g-2<0$, here $c_1(X_j)$ is the first Chern class of $X_j$ and $g$ is the genus of $C$. Moreover, assume that $C$ is not the unique effective curve in its cohomology class;

or 

3) A blowup of $X_j$ along a smooth curve $C$ contained in an irreducible hypersurface $S$ of $X_j$ so that $2\kappa <\mu \gamma$. Here $\kappa =S.C$, $1\leq \mu =$ the multiplicty of $C$ in $S$, and $\gamma =c_1(X_j).C+2g-2$ is the same as in 2).

Then, for every $f\in Aut(X)$, its dynamical degrees satisfy $\lambda _1(f)=\lambda _2(f)=1$. Therefore, by Gromov-Yomdin's theorem (see Theorem \ref{TheoremGromovYomdin} below), $h_{top}(f)=0$. 
\label{TheoremAutomorphismBlowupP3}\end{theorem}

As a consequence we obtain the following
\begin{corollary}
Let $X=X_n\rightarrow X_{n-1}\rightarrow \ldots \rightarrow X_1\rightarrow X_0=\mathbb{P}^3$ be a finite composition of blowups along smooth centers. Assume that each $X_{j+1}\rightarrow X_j$ ($1\leq j\leq n-1$) is either one of the cases i), ii), iii) in Theorem \ref{TheoremAutomorphismBlowupP3}. Then the automorphism group $Aut(X)$ has only finitely many connected components. 
\label{CorollaryAutomorphismGroupBlowupP3}\end{corollary}
Remark: The special case of finite point-blowups of $\mathbb{P}^3$ (i.e. all $X_{j+1}\rightarrow X_j$ is the case i) in Theorem \ref{TheoremAutomorphismBlowupP3}) was proved in the paper Bayraktar-Cantat \cite{bayraktar-cantat}.
\begin{proof}
In this case we can take $X_1=X_0=\mathbb{P}^3$ in Theorem \ref{TheoremAutomorphismBlowupP3}. The proof of Theorem \ref{TheoremAutomorphismBlowupP3} implies that if $\zeta$ is a non-zero nef class on $X$ then $\zeta .\zeta \not= 0$. Hence the proof of Theorem 1.1 in the paper \cite{bayraktar-cantat} implies that $Aut(X)$ has only finitely many connected components.  
\end{proof}

{\bf Remark:} 

-Even though we stated Theorems \ref{TheoremDynamicalDegreesAutomorphismsBlowupP3}, Corollary \ref{CorollaryAutomorphismGroupBlowupP3} and \ref{TheoremAutomorphismBlowupP3} only for $\mathbb{P}^3$, we can modify them to apply to other spaces, for example $\mathbb{P}^2\times \mathbb{P}^1$ or $\mathbb{P}^1\times \mathbb{P}^1\times \mathbb{P}^1$. Theorem \ref{TheoremDynamicalDegreesAutomorphismsBlowupP3} and Corollary \ref{CorollaryAutomorphismGroupBlowupP3} can be repeated verbatim, while the conclusions of Theorem \ref{TheoremAutomorphismBlowupP3} still hold if we do not include the space $X_1$ in the statement. This is necessary, since if $Z$ is an appropriate blowup of $\mathbb{P}^2$ at points in $\mathbb{P}^2$ then $Z$ has an automorphism $g$ of positive entropy (see McMullen \cite{mcmullen1}), therefore the space $Z\times \mathbb{P}^1$ has an automorphism of positive entropy as well. This space $Z\times \mathbb{P}^1$ is one of the spaces $X_1$ if we start from $X_0=\mathbb{P}^2\times \mathbb{P}^1$, yet it has an automorphism of positive entropy. We will show however that this example is compatible with our Theorem \ref{TheoremAutomorphismBlowupP3} in that the curves we blowup to form the space $Z\times \mathbb{P}^1$ do not satisfy both conditions 2) and 3) of Theorem \ref{TheoremAutomorphismBlowupP3}. We will also construct blowups of pairwise disjoint curves on $\mathbb{P}^2\times \mathbb{P}^1$ that do satisfy at least one of these conditions. For details please see the last section.        

-Conditions 2) and 3) in Theorem \ref{TheoremAutomorphismBlowupP3} are complement to each other: 2) is applied for $\gamma <0$ while 3) may be applied for $\gamma \geq 0$. The examples mentioned in the above paragraph show that conditions 2) and 3) are somewhat optimal, and there are cases when condition 2) does not apply while condition 3) does apply. 

-Consider condition 2) in Theorem \ref{TheoremAutomorphismBlowupP3}. Let $F\subset X_{j+1}$ be the exceptional divisor of the blowup $X_{j+1}\rightarrow X_j$, and let $M\subset F$ be a fiber of the projection $F\rightarrow C$. If there is a non-zero effective curve $V\subset F$ so that $F.V\geq 0$ then the requirement that $C$ is not the only effective curve in its cohomology class is not needed. For further comment on this, please see Lemma \ref{LemmaCaseTauNonNegative}.

-Consider condition 3) of Theorem \ref{TheoremAutomorphismBlowupP3}. Let $F\subset X_{j+1}$ be the exceptional divisor of the blowup $X_{j+1}\rightarrow X_j$, and let $M\subset F$ be a fiber of the projection $F\rightarrow C$. Then condition 3) implies the existence of an effective curve $C_0\subset F$ with the properties $C_0.C_0<0$ and $C_0.M>0$. Part e) of the proof of of Theorem \ref{TheoremAutomorphismBlowupP3} and Lemma \ref{LemmaIntersectionOnF} implies that Theorem \ref{TheoremAutomorphismBlowupP3} still holds if we replace condition 3) by the latter. For further comment on this please see Lemma \ref{LemmaCaseTauNonNegative}. 

-In condition 3) of Theorem \ref{TheoremAutomorphismBlowupP3}, given an irreducible curve $C\subset Y$, there is always a hypersurface $S\subset Y$ containing $C$. In fact, if $C$ is in the strict transform of an exceptional divisor then we can choose $S$ to be that hypersurface. Otherwise, $C$ is the strict transform of some curve $D\subset \mathbb{P}^3$. In this case we choose $S$ to be the strict transform of a hypersurface in $\mathbb{P}^3$ containing $D$.

-It is a natural and important problem to understand finer structures of the automorphisms of blowups of $\mathbb{P}^3$ (such as whether the automorphism groups of a generic space constructed in Theorem \ref{TheoremAutomorphismBlowupP3} is trivial, see also Question 1 below). Unfortunately we are not able to answer this in the current paper. 

There are many examples realizing the conditions of Theorem \ref{TheoremAutomorphismBlowupP3} (and Theorem  \ref{TheoremDynamicalDegreesAutomorphismsBlowupP3}). 

{\bf Example 1:} For condition 1), we blowup a point in $X_j$ and then we can blowup any number of points and curves on the exceptional divisor, and then can do iterated blowups on the resulting exceptional divisors and so on.

{\bf Example 2:} For condition 2), assume that we have a smooth curve $D$ on $X_j$ and another effective curve $D'$ in the cohomology class of $D$ so that $D$ and $D'$ intersect in a large enough number of points (counted with multiplicities). Then we blowup these intersection points (may need to do iterated blowup when the multiplicity is greater than $1$), and then blowup the strict transform of $D$. Another way of constructing is to blowup many curves having non-empty intersections with $D$ (see Example 6). 

{\bf Example 3:} Let $D$ be a smooth curve of degree $d\geq 2$ contained in a hyperplane $W$ of $\mathbb{P}^3$. If $Y\rightarrow \mathbb{P}^3$ is the blowup of $t$ points on $D$ (for any number $t$), and $C$ is the strict transform of $D$ then condition 3) is satisfied if we choose $S$ to be the strict transform of the hyperplane containing $D$. If in contrast, $D$ has degree $1$ (and therefore is a projective line), then we can apply Theorem \ref{TheoremDynamicalDegreesAutomorphismsBlowupP3} provided $t\not= 3$ (See Example 5 also).

{\bf Example 4:} Let $C_1$ and $C_2$ be two smooth curves, both belonging to the same hyperplane $W\subset \mathbb{P}^3$. Let $Y\rightarrow \mathbb{P}^3$ be the blowup at $C_1$, and let $X\rightarrow Y$ be the blowup at the strict transform of $C_2$. Then any automorphism of $X$ has zero entropy.

{\bf Example 5}: Let $Y\rightarrow \mathbb{P}^3$ be the blowup of $\mathbb{P}^3$ at $4$ points $e_0=[1:0:0:0]$, $e_1=[0:1:0:0]$, $e_2=[0:0:1:0]$ and $e_3=[0:0:0:1]$. For $0\leq i\not= j\leq 3$, let $\Sigma _{i,j}\subset \mathbb{P}^3$ be the line connecting $e_i$ and $e_j$. Let $\widetilde{\Sigma _{i,j}}$ be the strict transform in $Y$ of $\Sigma _{i,j}$. The curves $\widetilde{\Sigma _{i,j}}$ are pairwise disjoint. Let $X\rightarrow Y$ be the blowup of $Y$ along the curves $\widetilde{\Sigma _{i,j}}$. Then for every element $f$ of $Aut(X)$ we have $\lambda _1(f)=\lambda _2(f)$. (See Example 3 also).

{\bf Example 6}: Notations are as in Example 5. Let $X\rightarrow \mathbb{P}^3$ be the blowup of $\mathbb{P}^3$ at $e_1$ and $e_3$, followed by blowup of the strict transform of $\Sigma _{0,1}$ and then blowup of the strict transform of $\Sigma _{0,3}$. Then any automorphism of $X$ has zero topological entropy. Bedford and Kim \cite{bedford-kim3} constructed this space in connection with pseudo-automorphic linear fractional maps. Theorem \ref{TheoremAutomorphismBlowupP3} does not apply directly to this example but we can adapt the proof to it.

Theorem \ref{TheoremAutomorphismBlowupP3} gives support to the guess that the answer to the following question, asked by Professor Eric Bedford in a conference in Paris in Jun 2011, is No:

{\bf Question 1}: Is there a finite blowup $\pi :X\rightarrow \mathbb{P}^3$ and an automorphism $f:X\rightarrow X$ with $h_{top}(f)>0$?

We end this section giving a heuristic argument to explain why there are few automorphisms of positive entropies of projective (or more generally, compact K\"ahler) manifolds $X$ of dimension $3$ (and higher dimensions). Let $f\in Aut(X)$ and choose $\eta $ a non-zero nef-class (see the next section for definition of nef-classes) which is an eigenvector of eigenvalue $\lambda _1(f)$ of the linear map $f^*:H^{1,1}(X)\rightarrow H^{1,1}(X)$ (the existence of such an $\eta $ is assured by a Perron-Frobenius type theorem, see the proof of Theorem \ref{TheoremNonExistenceClassB}). As the proof of Theorem \ref{TheoremNonExistenceClassB} below shows, if for every $f\in Aut(X)$ we can choose such an $\eta $ so that $\eta .\eta \not= 0$ then every automorphism of $X$ has zero entropy. It is very unlikely to have $\eta .\eta =0$. In fact, by Poincare duality and Hodge decomposition, $dim (H^{1,1}(X))=dim(H^{2,2}(X))$. Denote by $n$ this dimension, let $x_1,\ldots ,x_n$ be a basis for $H^{1,1}(X)$ and let $y_1,\ldots ,y_n$ be a basis for $H^{2,2}(X)$. Then there are numbers $a_1,\ldots ,a_n$ so that $\eta =a_1x_1+\ldots +a_nx_n$. We can write $\eta .\eta =P_1(a_1,\ldots ,a_n)y_1+\ldots +P_n(a_1,\ldots ,a_n)y_n$, here $P_1(a_1,\ldots ,a_n),\ldots ,P_n(a_1,\ldots ,a_n)$ are homogeneous polynomials of degree $2$ in the variables $a_1,\ldots ,a_n$. The coefficients of these polynomials depend only on the intersection product on the cohomology groups of $X$. If $\eta .\eta =0$, then $P_1(a_1,\ldots ,a_n)=\ldots =P_n(a_1,\ldots ,a_n)=0$. The latter, being a system of $n$ homogeneous equations in $n$ variables, is expected to have only the solution $a_1=\ldots =a_n=0$, even when we do not take into account the fact that $\eta$ is nef and is an eigenvector of eigenvalue $\lambda _1(f)$ of $f^*$. 

{\bf Remarks.}

1. In recent works Bayraktar \cite{bayraktar} and Bayraktar-Cantat \cite{bayraktar-cantat}, the authors considered a more refined condition of that used in the above heuristic argument. More precisely, they considered the manifolds $X$ such that for  any non-zero nef class $\zeta \in H^{1,1}(X)$ then $\zeta ^{k-r+1}$ is non-zero, here $r$ is a fixed integer with $k>2r+2$. 

2. Combined with the results in \cite{bayraktar-cantat}, the above heuristic argument can be applied to prove a stronger conclusion: For "generic" compact K\"ahler manifolds $X$ of dimension $\geq 3$, the automorphism group $Aut(X)$ has only finitely many connected components.     

{\bf Acknowledgements.} The author is grateful to Professor Tien-Cuong Dinh for suggesting that the answer to Question 1 should be No. He thanks Professor Mattias Jonsson for checking an earlier version of this paper, for his useful comments and interest in the topic of the paper; and thanks Professor Viet-Anh Nguyen for checking an earlier version of this paper and for useful discussions. He thanks Professors Ekaterina Amerik and Eric Bedford for helpful discussions. He is also thankful to Professors Laura DeMarco, Roland Roeder, and Joseph Silverman for their interest in the topic of the paper.  

\section{Preliminaries on positive cohomology classes, blowups, dynamical degrees, and entropies}

\subsection{K\"ahler, nef and psef classes, and effective varieties}

Let $X$ be a compact K\"ahler manifold. Let $\eta \in H^{1,1}(X)$. We say that $\eta $ is K\"ahler if it can be represented by a K\"ahler $(1,1)$ form. We say that $\eta $ is nef if it is a limit of a sequence of K\"ahler classes. We say that $\eta$ is psef if it can be represented by a positive closed $(1,1)$ current. A class $\xi \in H^{p,p}(X)$ is an effective variety if there are irreducible varieties $C_1,\ldots ,C_t$ of codimension $p$ in $X$ and non-negative real numbers $a_1,\ldots ,a_t$ so that $\xi$ is represented by $\sum _{i}a_iC_i$.

Demailly and Paun \cite{demailly-paun} gave a characterization of K\"ahler and nef classes, which in the case of projective manifolds is summarized as follows:
\begin{theorem}
Let $X$ be a projective manifold with a K\"ahler $(1,1)$ form $\omega$. A class $\eta \in H^{1,1}(X)$ is K\"ahler if and only for any irreducible subvariety $V\subset X$ then $\int _V\eta ^{dim(V)}>0$.  A class $\eta \in H^{1,1}(X)$ is nef if and only for any irreducible subvariety $V\subset X$ then $\int _V\eta ^{dim(V)-j}\wedge \omega ^j\geq 0$ for all $0\leq j\leq dim (V)$.  
\label{TheoremDemaillyPaun}\end{theorem}

Nef classes are preserved under pullback by holomorphic maps. 
\begin{lemma}
Let $\pi :X\rightarrow Y$ be a holomorphic map between compact K\"ahler manifolds. Then $\pi ^*(H^{1,1}_{nef}(X))\subset H^{1,1}_{nef}(Y)$.
\label{LemmaPullbackNefClasses}\end{lemma}
\begin{proof}
Since nef classes are in the closure of K\"ahler classes, it suffices to show that if $\eta$ is a K\"ahler class then $\pi ^*(\eta )$ is nef. Let $\varphi $ be a K\"ahler $(1,1)$ form representing $\eta$. Then $\pi ^*(\varphi )$ is a positive smooth $(1,1)$ form. Let $\omega _X$ be a K\"ahler $(1,1)$ form on $X$. Then $\pi ^*(\eta )$ is represented as a limit of the following K\"ahler classes
\begin{eqnarray*} 
\pi ^*(\varphi )+\frac{1}{n}\omega _X,
\end{eqnarray*}
and hence is nef.
\end{proof}

{\bf Remark:} Similarly, it can be shown that psef classes are preserved under pushforward by holomorphic maps. However, nef classes may not be preserved under pushforwards, even when the map is a blowup.  

\subsection{Blowup of a projective $3$-manifold at a point}

Let $\pi :X\rightarrow Y$ be the blowup of a projective $3$-manifold at a point $p$. 
Let $E= \mathbb{P}^2$ be the exceptional divisor and let $L\subset E$ be a line. Then $H^{1,1}(X)$ is generated by $\pi ^*(H^{1,1}(Y))$ and $E$, and $H^{2,2}(X)$ is generated by $\pi ^*(H^{2,2}(Y))$ and $L$. The intersection product on the cohomology of $X$ is given by 
\begin{eqnarray*}
\pi ^*(\xi ).E=0,~E.E=-L,\\
\pi ^*(\xi ).L=0,~E.L=-1. 
\end{eqnarray*}

The first and second Chern classes of $X$ can be computed by (see e.g. Section 6, Chapter 4 in the book of Griffiths-Harris \cite{griffiths-harris})
\begin{eqnarray*}
c_1(X)&=&\pi ^*(c_1(Y))-2E,\\
c_2(X)&=&\pi ^*(c_2(Y)).
\end{eqnarray*}

The following result concerns the relations between cycles on $X$ and $Y$.
\begin{lemma}
 For any effective curve $V\subset Y$, there is an effective curve $\widetilde{V}\subset X$ so that $\pi _*(\widetilde{V})=V$ and $\widetilde{V}.E\geq 0$.  
\label{LemmaPushforwardCyclesBlowupPoint}\end{lemma}  
\begin{proof}
It suffices to consider the case when $V$ is an irreducible curve. We can choose $\widetilde{V}$ to be the strict transform of $V$. Then $\pi _*(\widetilde{V})=V$, and $\widetilde{V}$ is not contained in $E$. Therefore $\widetilde{V}.E\geq 0$.  
\end{proof}

We end this subsection showing that nef classes are preserved under pushforward by point-blowups. 
\begin{lemma}
Let $\eta \in H^{1,1}_{nef}(X)$. Then $\pi _*(\eta )\in H^{1,1}_{nef}(Y)$.
\label{LemmaPushforwardNefBlowupPoint}\end{lemma}
\begin{proof}
It suffices to prove the conclusion when $\eta$ is a K\"ahler class. Let $\varphi$ be a K\"ahler $(1,1)$ form representing $\eta$. Then $\pi _*(\varphi )$ is a positive closed $(1,1)$ current, which is smooth on $X-p$. 

Let $\omega _Y$ be a K\"ahler $(1,1)$ form on $Y$. To show that $\pi _*(\eta )$ is a nef class, by Theorem \ref{TheoremDemaillyPaun} it suffices to show that for any irreducible variety $V\subset Y$ then $\pi _*(\eta )^{dim(V)-j}.V.\omega _Y^j\geq 0$ for $0\leq j\leq dim (V)$. We let $[V]$ be the current of integration on $V$. Then by the results in Section 4, Chapter 3 in the book of Demailly \cite{demailly}, the current $\pi _*(\varphi )^{dim(V)-j}\wedge [V]\wedge \omega _Y^j$ is well-defined and is a positive measure, whose mass equals to $\pi _*(\eta )^{dim(V)-j}.V.\omega _Y^j$. Thus the latter quantity is non-negative.  
\end{proof}

\subsection{Blowup of a projective $3$-manifold along a smooth curve}
Let $\pi :X\rightarrow Y$ be the blowup of a projective $3$-manifold along a smooth curve $C\subset Y$. Let $g$ be the genus of $C$. Let $F$ be the exceptional divisor and let $M$ be a fiber of the projection $F\rightarrow C$. We can identify $F$ with the projective bundle $\mathbb{P}(\mathcal{E})\rightarrow C$, where $\mathcal{E}=N_{C/Y}\rightarrow C$ is the normal vector bundle of $C$ in $Y$.   

Then $H^{1,1}(X)$ is generated by $\pi ^*(H^{1,1}(Y))$ and $F$, and $H^{2,2}(X)$ is generated by $\pi ^*(H^{2,2}(Y))$ and $M$. The intersection between $F$ and $M$ is $F.M=-1$. The first and second Chern classes of $X$ can be computed as follows:
\begin{eqnarray*}
c_1(X)&=&\pi ^*(c_1(Y))-F,\\
c_2(X)&=&\pi ^*(c_2(Y)+C)-\pi ^*c_1(Y).F.
\end{eqnarray*}

Let $[F]\rightarrow X$ be the line bundle of $F$ in $X$, and denote by $e=[F]|_F$. Then (see e.g. Section 6, Chapter 4 in the book of Griffiths - Harris \cite{griffiths-harris}) in $F$ we have the equalities 
\begin{eqnarray*}
e.M=-1,~e.e=-c_1(\mathcal{E}).
\end{eqnarray*}  
From the SES of vector bundles on $C$ 
\begin{eqnarray*}
0\rightarrow T_C\rightarrow T_Y|_C\rightarrow \mathcal{E}\rightarrow 0,
\end{eqnarray*}
it follows by the additivity of first Chern classes that 
\begin{eqnarray*}
c_1(\mathcal{E})=c_1(T_Y).C-c_1(T_C)=c_1(Y).C+2g-2.
\end{eqnarray*}
We define 
\begin{eqnarray*}
\gamma :=c_1(Y).C+2g-2.
\end{eqnarray*} 
Since $F\rightarrow C$ is a ruled surface (i.e. its fibers are projective lines $\mathbb{P}^1$), there is a canonical section $C_0$ which is the image of a holomorphic map $\sigma _0:C\rightarrow F$ (see e.g. Section 2, Chapter 5 in Hartshorne's book \cite{hartshorne}). Therefore $C_0$ is an effective curve in $F$. Such a $C_0$ has intersection $1$ with a fiber $M$.

We will return to the canonical section $C_0$ at the end of this subsection. For now, we however work in a more general assumption on $C_0$, for using later. That is, we consider an effective curve $C_0\subset F$ with the following properties  
\begin{eqnarray*}
C_0.C_0&=&\tau ,\\
C_0.M&=&\mu >0,\\
M.M&=&0.
\end{eqnarray*}   
Any divisor on $F$ is numerically equivalent to a linear combination of $C_0$ and $M$. We now show the following 

\begin{lemma}

a) 
\begin{equation}
F.C_0=\frac{1}{2}(\gamma \mu -\frac{\tau}{\mu}).
\label{Equation1}\end{equation}

b) \begin{eqnarray*}
F.F=-\frac{1}{\mu}C_0+\frac{1}{2}(\frac{\tau}{\mu ^2}+\gamma )M.
\end{eqnarray*}

c) $\pi _*(F.F)=-C$.
\label{LemmaIntersectionOnF}\end{lemma}

\begin{proof} 

a) In fact, we have
\begin{eqnarray*}
F.C_0=[F]|_{C_0}=[F]|_{F}.C_0=e.C_0,
\end{eqnarray*}
here the two expressions on the RHS are computed in $F$. On $F$, numerically we can write $e=aC_0+bM$. Then from $-1=e.M=(aC_0+bM).M=a\mu$, we get $a=-1/\mu$. Substitute  this into $e.e=-\gamma$ we obtain 
\begin{eqnarray*}
-\gamma =e.e=(\frac{1}{\mu}C_0-bM).(\frac{1}{\mu}C_0-bM)=\frac{\tau}{\mu ^2}-2b, 
\end{eqnarray*} 
which implies that 
\begin{eqnarray*}
b=\frac{1}{2}(\frac{\tau}{\mu ^2}+\gamma ). 
\end{eqnarray*}

Therefore 
\begin{eqnarray*}
e=\frac{-1}{\mu}C_0+\frac{1}{2}(\frac{\tau}{\mu ^2}+\gamma )M.
\end{eqnarray*}
Thus
\begin{eqnarray*}
F.C_0&=&e.C_0=[\frac{-1}{\mu}C_0+\frac{1}{2}(\frac{\tau}{\mu ^2}+\gamma )M]C_0\\
&=&\frac{-\tau}{\mu}+\frac{1}{2}(\frac{\tau}{\mu}+\gamma \mu )\\
&=&\frac{1}{2}(-\frac{\tau}{\mu}+\gamma \mu ).
\end{eqnarray*}

b) From the formula for $e$ in the proof of a) it is not difficult to arrive at the proof of b).

c) Since $C_0.M=\mu$, it follows that $\pi _*(C_0)=\mu C$. Then from b) we obtain c). 
\end{proof}

We end this subsection commenting on conditions 2) and 3) of Theorem \ref{TheoremAutomorphismBlowupP3}. By Proposition 2.8 in Chapter 5 of \cite{hartshorne}, there is a line bundle $\mathcal{M}\rightarrow C$ so that the vector bundle $\mathcal{E}'=\mathcal{E}\otimes \mathcal{M}$ is normalized in the following sense: $H^0(\mathcal{E}')\not= 0$ but for all line bundle $\mathcal{L}\rightarrow C$ with $c_1(\mathcal{L})<0$ then $H^0(\mathcal{E}'\otimes \mathcal{L})=0$. A canonical section $C_0\subset F$ can be associated to such a normalized $\mathcal{E}'$. The intersection between $C_0$ and $M$ is $1$. Moreover, the number 
\begin{eqnarray*}
\tau _0=C_0.C_0=c_1(\mathcal{E}')=c_1(\mathcal{E})+2c_1(\mathcal{M}),   
\end{eqnarray*}
is an invariant of $F$.

Condition 3) of Theorem \ref{TheoremAutomorphismBlowupP3} implies the existence of an effective curve $V\subset F$ for which $V.V<0$ and $V.M>0$. We now show that such an effective curve exists if and only if the invariant $\tau _0$ is $<0$.   In condition 2) of Theorem \ref{TheoremAutomorphismBlowupP3}, the requirement that $C$ is not the only effective curve in its cohomology class is not needed if there exists a non-zero effective curve $V\subset F$ so that $F.V\geq 0$. We now also show that if $\gamma <0$ and $\tau _0\geq 0$ then such a curve $V$ does not exist. 

\begin{lemma}
Assume that the invariant $\tau _0$ of $F$ is non-negative. Then

a) For any effective curve $V\subset F$ we have $V.V\geq 0$.

b) If moreover $\gamma <0$ then for any non-zero effective curve $V\subset F$ we have $F.V<0$.
\label{LemmaCaseTauNonNegative}\end{lemma}  
\begin{proof}

a) It suffices to prove for the case $V$ is an irreducible curve. Numerically, we write $V=aC_0+bM$. If $V=C_0$ then $V.V=\tau _0\geq 0$. If $V=M$ then $V.V=0$. Hence we may assume that $V\not= C_0,M$. 

We consider two cases:

Case 1: $\tau _0=0$. By Proposition 2.20 in Chapter 5 in \cite{hartshorne}, we have $a>0$ and $b\geq 0$. Therefore
\begin{eqnarray*}
 V.V=a^2\tau _0+2ab\geq 0.
\end{eqnarray*}

Case 2: $\tau _0>0$. By Proposition 2.21 in Chapter 5 in \cite{hartshorne}, there are two subcases:

Subcase 2.1: $a=1,b\geq 0$. Then 
\begin{eqnarray*}
V.V=\tau _0+2b\geq 0.
\end{eqnarray*}

Subcase 2.2: $a\geq 2,b\geq -a\tau _0/2$. Then
\begin{eqnarray*}
V.V=a^2\tau _0+2ab\geq a^2\tau _0+2a(-a\tau _0/2)=0.
\end{eqnarray*}

b) It suffices to prove for the case $V$ is an irreducible curve. If $V=M$ then $F.M=-1<0$. If $V=C_0$ then by Lemma \ref{LemmaIntersectionOnF} with $\tau =\tau _0\geq 0$ and $\mu =1$
\begin{eqnarray*}
F.C_0=\frac{1}{2}(\gamma -\tau _0)\leq \frac{1}{2}\gamma <0
\end{eqnarray*}
because $\gamma <0$. Therefore we may assume that $V\not= C_0,M$, and then proceed as in the proof of a). 
\end{proof}

\subsection{Dynamical degrees and entropy} Let $f:X\rightarrow X$ be a surjective holomorphic map of a compact K\"ahler manifold of dimension $k$. For $1\leq p\leq k$, we define the $p$-th dynamical degree $\lambda _p(f)$ of $f$ to be the spectral radius of the linear map $f^*:H^{p,p}(X)\rightarrow H^{p,p}(X)$. The dynamical degrees are all $\geq 1$, and are log-concave, i.e. $\lambda _{j}(f)^2\geq \lambda _{j+1}(f)\lambda _{j-1}(f)$ (see Dinh-Sibony \cite{dinh-sibony1}\cite{dinh-sibony2}).

Let $d$ be a metric on $X$. A subset $E$ of $X$ is called $(n,\epsilon )$-separated if for any pair $x\not= y\in E$ then $\max\{d(f^i(x),f^i(y)):~0\leq i\leq n-1\}\geq \epsilon$. Denote by $N(n,\epsilon )$ the maximal cardinality of an $(n,\epsilon )$-separated set. Then the topological entropy of $f$ is given by
\begin{eqnarray*}   
h_{top}(f)=\lim _{\epsilon \rightarrow 0}\limsup _{n\rightarrow\infty}\frac{1}{n}\log N(n,\epsilon ).
\end{eqnarray*}

Gromov \cite{gromov} and Yomdin \cite{yomdin} proved the following result, relating dynamical degrees to topological entropy:
\begin{theorem}
Assumptions as above. Then $h_{top}(f)=\max _{1\leq p\leq k}\log \lambda _p(f)$. 
\label{TheoremGromovYomdin}\end{theorem}

Apply the log concavity of dynamical degrees to Gromov-Yomdin's theorem, we deduce that $h_{top}(f)=0$ if and only if $f$ is an automorphism and $\lambda _p(f)=1$ for some (and hence, all) $1\leq p\leq k-1$.

\section{Proofs of Theorems \ref{TheoremDynamicalDegreesAutomorphismsBlowupP3} and  \ref{TheoremAutomorphismBlowupP3}}
For the proof of Theorems \ref{TheoremDynamicalDegreesAutomorphismsBlowupP3} and \ref{TheoremAutomorphismBlowupP3}, we first introduce the following set of cohomology classes, which uses a weaker notion of positivity than that of nef classes. 

\begin{definition} Let $X$ be a projective manifold of dimension $3$. We define by $\mathcal{B}(X)$ the set of cohomology classes $\eta \in H^{1,1}(X)$ satisfying the following conditions:

1) $\eta $ is psef.

2) For every effective curve $V$ in $X$ then $\eta .V\geq 0$.
\label{DefinitionClassB}\end{definition}

We also introduce a larger set of cohomology classes
\begin{definition} Let $X$ be a projective manifold of dimension $3$. We define by $\mathcal{C}(X)$ the set of cohomology classes $\eta \in H^{1,1}(X)$ satisfying the following conditions:

1) $\eta $ is psef.

2) There is a finite number of irreducible curves $V_1,\ldots ,V_t\subset X$ (these curves depend on $\eta $) so that if $V$ is an irreducible curve in $X$ with  $\eta .V< 0$, then $V$ is one of the curves $V_1,\ldots ,V_t$. 
\label{DefinitionClassB}\end{definition}

We have obvious inclusions $H^{1,1}_{nef}(X)\subset \mathcal{B}(X)\subset \mathcal{C}(X)$. The following properties of $\mathcal{B}(X)$ and $\mathcal{C}(X)$ make them useful in induction arguments involving finite blowups in dimension $3$. 

\begin{lemma}
Let $\pi :X\rightarrow Y$ be a blow up of a projective $3$-manifold $Y$ along a point $p\in Y$ or a smooth curve $C\subset Y$. 

a) If $\eta $ is in $\mathcal{C}(X)$ then $\pi _*(\eta )$ is in $\mathcal{C}(Y)$.

b) If $\pi$ is a point blowup and $\eta \in \mathcal{B}(X)$ then $\pi _*(\eta )\in \mathcal{B}(Y)$.

c) If $\pi $ be a blowup of $Y$ along a smooth curve $C\subset Y$ so that $C$ is not the only effective curve in its cohomology class, then $\pi _*(\mathcal{B}(X))\subset \mathcal{B}(Y)$. 
\label{LemmaPushforwardClassBBlowup}\end{lemma}
\begin{proof}

b) Consider the case when $\pi :X\rightarrow Y$ is a point blowup. Let $E=\mathbb{P}^2$ be the exceptional divisor and let $L\subset E$ be a line. Let $\eta \in \mathcal{B}(X)$ and let $\xi =\pi _*(\eta )$. Then $\xi$ is psef since $\eta$ is so. Let $V\subset Y$ be an irreducible curve, and let $\widetilde{V}\subset X$ be the strict transform of $V$. Then $\pi _*(\widetilde{V})=V$, and $\widetilde{V}$ is not contained in $\widetilde{E}$ therefore $\widetilde{V}.E\geq  0$.  

We can write $\eta =\pi ^*(\xi )-\alpha E$ where $\alpha =\eta .L\geq 0$ since $\eta \in \mathcal{B}(X)$. Hence
\begin{eqnarray*}
\xi .V=\xi .\pi _*(\widetilde{V})=\pi ^*(\xi ).\widetilde{V}=(\eta +\alpha E).\widetilde{V}\geq 0.
\end{eqnarray*}
 
Thus $\xi \in \mathcal{B}(X)$.

a) We consider first the case when $\pi :X\rightarrow Y$ is a blowup of $Y$ along a smooth curve $C\subset Y$. Let $F$ be the exceptional divisor of the blowup. Let $M$ be a fiber of $F\rightarrow C$. Let $\eta \in \mathcal{C}(X)$ and let $\xi =\pi _*(\eta )$. Then $\eta =\pi ^*(\xi )-\alpha F$, where $\alpha =\eta .M$. Observe that $\alpha \geq 0$, because $\eta $ can have negative intersections with only a finite number of irreducible curves while we have infinitely many fibers.

Since $\eta $ is psef, it follows that $\xi $ is psef as well. Let $V\subset Y$ be an irreducible curve which is not contained in the union of $C$ with the images of the irreducible curves having negative intersections with $\eta $. Then we can proceed as in the proof of b) to show that $\xi .V\geq 0$. Hence $\xi \in \mathcal{C}(X)$.

The proof of the case $\pi$ is a point blowup is similar. 

c) Let $\pi :X\rightarrow Y$ is a blowup of $Y$ along a smooth curve $C\subset Y$, where $C$ is not the only effective curve in its cohomology class. Let $F$ be the exceptional divisor of the blowup. Let $M$ be a fiber of $F\rightarrow C$. Let $\eta \in \mathcal{C}(X)$ and let $\xi =\pi _*(\eta )$. Then $\eta =\pi ^*(\xi )-\alpha F$, where $\alpha =\eta .M\geq 0$. 

If $V\subset Y$ is an irreducible curve different from $C$, then by using its strict transform in $X$ we can show as in the proof of b) that $\xi .V\geq 0$. If $V=C$, then since $C$ is not the only effective curve in its cohomology class, we can find an effective curve $C'$ having the same cohomology class as that of $C$ so that the support of $C'$ does not contain $C$. Then we can proceed as in the first case.  
\end{proof}

\begin{lemma}
Let $\pi :X\rightarrow Y$ be a finite composition $X=X_n\rightarrow X_{n-1}\rightarrow \ldots \rightarrow X_1\rightarrow Y$ of point or curve blowups. Assume that the first map $X_1\rightarrow Y$ is a point blowup. Assume moreover that the images of the exceptional divisors of the map $X\rightarrow X_1$ are contained in the exceptional divisor of $X_1\rightarrow Y$. 

Let $\eta \in \mathcal{C}(X)$ and let $V\subset X$ be an effective curve such that $\eta .\eta =V$ and $\pi _*(V)=0$. Then $V=0$. If moreover $\eta \in \mathcal{B}(X)$ then the pushforward of $\eta $ under the map $X\rightarrow X_j$ is in $\mathcal{B}(X_j)$. 
\label{LemmaSolveEquationEta2IsAnEffectiveDivisor}\end{lemma}
\begin{proof}
We prove this by induction on $n$.

The initial case $n=1$: Let $E_1=\mathbb{P}^2$ be the exceptional divisor of the map $p:X_1\rightarrow Y$. Then we need to show that if $V$ is an effective curve with support in $E_1$ and $\eta \in \mathcal{C}(X_1)$ so that $\eta .\eta =V$ then $V=0$. Let us denote $\xi =\pi _*(\eta )\in \mathcal{C}(Y)$. Let $L_1$ be a line in $E_1=\mathbb{P}^2$. Then $\eta =\pi ^*(\xi )-\alpha E_1$ where $\alpha =\eta .L_1\geq 0$ (as in the proof of a) of Lemma \ref{LemmaPushforwardClassBBlowup}). Since support of $V$ is in $E_1$ and $V$ is effective, there is a number $\beta \geq 0$ so that $V=\beta L_1$. Since $E_1.E_1=-L_1$ and $E_1.\pi ^*(\xi )=0$, we have 
\begin{eqnarray*}
\pi ^*(\xi .\xi  )-\alpha ^2L_1=\eta .\eta =\beta L_1.
\end{eqnarray*}
Since $\pi ^*(\xi .\xi )$ and $L_1$ are linearly independent, it follows that $ -\alpha ^2=\beta $. Since both $\alpha $ and $\beta$ are non-negative, we get $\beta =0=\alpha$. Thus $V=0$ as claimed. If moreover $\eta \in \mathcal{B}(X)$, from the fact that $\eta =p ^*(\xi )$, it follows easily that $\xi \in \mathcal{B}(Y)$ as well.

Now assume that we had the claim for $n=j$. We will prove it for $n=j+1$. Let $p$ denote the map $X_{j+1}\rightarrow X_j$. Let $\eta \in \mathcal{C}(X_{j+1})$ and $V\subset X_{j+1}$ an effective curve so that $\eta .\eta =V$ and $\pi _*(V)=0$ in $H^{2,2}(Y)$. We need to show that $V=0$.    

We consider two cases:

Case 1: $p$ is a point blowup. Let $E=\mathbb{P}^2$ be the exceptional divisor of $p$, and let $L\subset E$ be a line. Let $\xi =p_*(\eta )\in \mathcal{C}(X_j)$, and write $\eta =p^*(\xi )-\alpha E$ for $\alpha =\eta .L\geq 0$. Then $\eta .\eta =V$ becomes
\begin{eqnarray*} 
p^*(\xi .\xi )-\alpha ^2L=V.
\end{eqnarray*}
Push-forward this equation by $p$ we obtain $\xi .\xi =p_*(V)$. Since the push-forward of $V$ under the map $X_{j+1}\rightarrow Y$ is zero, it follows that the push-forward of $p_*(V)$ under the map $X_j\rightarrow Y$ is zero. Therefore the induction assumption implies that $p_*(V)=0$. Therefore $V$ must be a multiple of $L$, and we can write $V=\beta L$ for some $\beta \geq 0$. Also $\xi .\xi =0$ and thus $p^*(\xi .\xi )=0$. Replace this into the original equation we get $-\alpha ^2L=\beta L$ which implies $\beta =\alpha =0$, i.e. $V=0$. If moreover $\eta \in \mathcal{B}(X_{j+1})$, from the fact that $\eta =p ^*(\xi )$, it follows easily that $\xi \in \mathcal{B}(X_j)$ as well.

Case 2: $p$ is a blowup of a smooth curve $C\subset X_j$ so that the push-forward of $C$ under the map $X_j\rightarrow Y$ is $0$. Let $F$ be the exceptional divisor of $p$ and let $M\subset F$ be a fiber of the map $F\rightarrow C$. Let $\xi =p_*(\eta )\in \mathcal{C}(X_j)$, and write $\eta =p^*(\xi )-\alpha F$ for $\alpha =\eta .M\geq 0$. Then $\eta .\eta =V$ becomes
\begin{eqnarray*} 
p^*(\xi .\xi )-2\alpha \pi ^*(\xi ).F+\alpha ^2(F.F)=V.
\end{eqnarray*}
Push-forward this equation by $p$ we obtain $\xi .\xi =\alpha ^2C+p_*(V)$. Since the push-forward of $V$ under the map $X_{j+1}\rightarrow Y$ is zero, it follows that the push-forward of $p_*(V)$ under the map $X_j\rightarrow Y$ is zero. Therefore, the class $\alpha ^2C+p_*(V)$ is effective and has image zero under push-forward by the map $X_j\rightarrow Y$. Apply the induction assumption we have that $\xi .\xi =0=\alpha ^2C+p _*(V)$ and hence $\alpha =0$. The original equation becomes $0=V$, and we are done. If moreover $\eta \in \mathcal{B}(X_{j+1})$, from the fact that $\eta =p ^*(\xi )$ and the existence of a section $C_0\subset F$ (see Section 2.3), we have $\xi .C=\eta .C_0\geq 0$ and it easily follows that $\xi \in \mathcal{B}(Y)$ as well.     
\end{proof}

Now we prove a general result on non-existence of automorphisms of positive entropies (see also Lemma 2.4 and other results in Zhang \cite{zhang3}, and  Dinh-Sibony \cite{dinh-sibony}). 

\begin{theorem}
Let $X$ be a projective manifold of dimension $3$ and let $f:X\rightarrow X$ be an automorphism. Assume that whenever $\eta \in H^{1,1}_{nef}(X)$ is an eigenvector of $f^*:H^{1,1}(X)\rightarrow H^{1,1}(X)$ then either $\eta .\eta \not= 0$ or $\eta \in \mathbb{R}.H^{1,1}(X,\mathbb{Q})$, i.e. there is a real number $a$ and a class $\eta _0\in H^{1,1}(X,\mathbb{Q})$ so that  $\eta =a\eta _0$ (in other words, $\eta$ is proportional to a rational cohomology class and hence to an integral class). Then $\lambda _1(f)=\lambda _2(f)=1$, and therefore $h_{top}(f)=0$.
\label{TheoremNonExistenceClassB}\end{theorem}
\begin{proof}
Assume in contrast that $\lambda _1(f)>1$. Since $f^*$ preserves the cone $H^{1,1}_{nef}(X)$, by a Perron-Frobenius type theorem, there is a non-zero nef class $\eta $ so that $f^*(\eta )=\lambda _1(f)\eta $. 

First we claim that for such an $\eta $, then $\eta .\eta \not= 0$. Otherwise, by assumption we can write $\eta =a\eta _0$ for some real number $a\in \mathbb{R}$ and $\eta _0\in H^{1,1}(X,\mathbb{Q})$. Dividing by $a$ we may assume that $\eta =\eta _0$ is in $H^{1,1}(X,\mathbb{Q})$. Since $f^*$ preserves $H^{1,1}(X,\mathbb{Q})$, it follows from $f^*(\eta )=\lambda _1(f)\eta $ that $\lambda _1(f)\in \mathbb{Q}$. However, the latter is irrational (see e.g. Zhang \cite{zhang3} and Bedford \cite{bedford}). [For the convenience of the readers, we reproduce the proof of this fact here. Let $A$ be the matrix of $f^*:H^{2}(X,\mathbb{C})\rightarrow H^{2}(X,\mathbb{C})$, then $A$ is an integer matrix, and $\lambda _1(f)$ is a real eigenvalue of $A$. Moreover, $A$ is invertible and its inverse $A^{-1}$ is the matrix of the map $(f^{-1})^*:H^{2}(X,\mathbb{C})\rightarrow H^{2}(X,\mathbb{C})$ hence is also an integer matrix. Therefore $det(A)=\pm 1$. Thus the characteristic polynomial $P(x)$ of $A$ is a monic polynomial of integer coefficients and $P(0)=\pm 1$. Assume that $\lambda _1(f)$ is a rational number. Since $\lambda _1(f)$ is an algebraic integer, it follows that $\lambda _1(f)$ must be an integer. Then we can write $P(x)=(x-\lambda _1(f))Q(x)$, here $Q(x)$ is a polynomial of integer coefficients. If $\lambda _1(f)>1$ we get a contradiction $\pm 1=P(0)=-\lambda _1(f)Q(0)$]. Thus $\eta .\eta \not= 0$ as claimed. 

Therefore $\eta .\eta $ is an eigenvector of $f^*:H^{2,2}(X)\rightarrow H^{2,2}(X)$ of eigenvalue $\lambda _1(f)^2$. Hence $\lambda _2(f)\geq \lambda _1(f)^2$. Since eigenvectors of $f^*:H^{1,1}(X)\rightarrow H^{1,1}(X)$ and $(f^{-1})^*:H^{1,1}(X)\rightarrow (f^{-1})^*:H^{1,1}(X)\rightarrow H^{1,1}(X)$ are the same (the operators $f^*$ and $(f^{-1})^*$ are inverse to each other), we can apply the same argument to the inverse $f^{-1}$ to obtain $\lambda _2(f^{-1})\geq \lambda _1(f^{-1})^2$. But $\lambda _1(f^{-1})=\lambda _2(f)$ and $\lambda _2(f^{-1})=\lambda _1(f)$, since $(f^{-1})^*=f_*$ is conjugate to $f^*$. In fact, let $\omega _X$ be a K\"ahler $(1,1)$ form on $X$. Then (see Dinh-Sibony \cite{dinh-sibony2}\cite{dinh-sibony1})
\begin{eqnarray*}
\lambda _1(f^{-1})&=&\lim _{j\rightarrow\infty}(\int _X(f^{-j})^*(\omega _X)\wedge \omega _X^2)^{1/j}=\lim _{j\rightarrow\infty}(\int _X(f^{j})_*(\omega _X)\wedge \omega _X^2)^{1/j}\\
&=&\lim _{j\rightarrow\infty}(\int _X\omega _X\wedge (f^j)^*(\omega _X^2))^{1/j}=\lambda _2(f),
\end{eqnarray*}
and similarly for the equality $\lambda _2(f^{-1})=\lambda _1(f)$.

Hence we must have $\lambda _1(f)=\lambda _1(f)^2=\lambda _2(f)=\lambda _2(f)^2=1$, as claimed. 
\end{proof}

Now we are ready to give the proofs of Theorems \ref{TheoremAutomorphismBlowupP3} and \ref{TheoremDynamicalDegreesAutomorphismsBlowupP3}.
\begin{proof} (Of Theorem \ref{TheoremAutomorphismBlowupP3}) Let $X=X_n\rightarrow X_{n-1}\rightarrow \ldots \rightarrow X_1\rightarrow X_0=\mathbb{P}^3$ be as in the statement of Theorem \ref{TheoremAutomorphismBlowupP3}. To prove Theorem \ref{TheoremAutomorphismBlowupP3}, it suffices to show that $X$ satisfies the conditions of Theorem \ref{TheoremNonExistenceClassB}. Indeed we will prove a stronger condition: 

Condition (A): If $\eta \in \mathcal{B}(X)$ satisfies $\eta .\eta =0$ then $\eta \in \mathbb{R}.H^{1,1}(X,\mathbb{Q})$.  

We  prove this by induction on $n$.

a) The initial step $n=0$ is clear, since then $X=X_0=\mathbb{P}^3$ and hence if $\eta \in H^{1,1}(X)$ is such that $\eta .\eta =0$ then $\eta =0$. 

b) We show that if $\pi :X=X_1\rightarrow \mathbb{P}^3$ is the blowup of points $p_1,\ldots ,p_t\in \mathbb{P}^3$ and smooth curves $C_1,\ldots ,C_s\subset \mathbb{P}^3$ in general positions, then $X$ satisfies Condition (A). Let $H\subset \mathbb{P}^3$ be the class of a generic hyperplane. Let $E_1,\ldots ,E_t$ be the exceptional divisors corresponding with $e_1,\ldots ,e_t$, and let $L_i\subset E_i$ be a line. Let $F_1,\ldots ,F_s$ be the exceptional divisors corresponding to $C_1,\ldots ,C_s$, and let $M_j\subset F_j$ be a fiber of the projection $F_j\rightarrow C_j$. Let $d_j\geq 1$ be the degree of $C_j$ (hence $d_j=H.C_j$ in $\mathbb{P}^3$), and let $g_j\geq 0$ be the genus of $C_j$. 

The cohomology group $H^{1,1}(X)$ is generated by $H,E_1,\ldots ,E_t,F_1,\ldots ,F_s$, and the cohomology group $H^{2,2}(X)$ is generated by $H.H,L_1,\ldots ,L_t,M_1,\ldots ,M_s$. The intersection product on $X$ is as follows (see, e.g., Section 6 Chapter 4 in the book of Griffiths and Harris \cite{griffiths-harris})
\begin{eqnarray*}
&&H.E_i=0,~H.F_i=d_iM_i,\\
&&E_i.E_j=-\delta _{i,j}L_i,~E_i.F_j=0.\\
&&F_i.F_j=\delta _{i,j}[-d_iH.H+(4d_i+2g_i-2)M_i].
\end{eqnarray*}

If $\eta \in H^{1,1}(X,\mathbb{R})$ we can write $\eta =aH-\sum _{i}e_iE_i-\sum _{j}f_jF_j$ for real numbers $a,e_1,\ldots ,e_t,f_1,\ldots ,f_s$. Then a computation shows 
\begin{eqnarray*}
&&\eta .\eta\\
 &=&a^2H.H-2aH(\sum _ie_iE_i)-2aH(\sum _jf_jF_j)+(\sum _ie_iE_i)^2+(\sum _jf_jF_j)^2+2(\sum _ie_iE_i).(\sum _jf_jF_j)\\
&=&a^2H.H-2a\sum _jd_jf_jM_j-\sum e_i^2L_i+\sum _{j}f_j^2[-d_jH^2+(4d_j+2g_j-2)M_j]. 
\end{eqnarray*}
Therefore $\eta .\eta =0$ if and only if 
\begin{eqnarray*}
a^2&=&\sum _jd_jf_j^2,\\
e_i^2&=&0,~\forall i=1,\ldots ,t,\\ 
2ad_jf_j&=&f_j^2(4d_j+2g_j-2),~\forall j=1, \ldots ,s. 
\end{eqnarray*}

From the equations for $e_i$ we have that $e_i=0$. If $a=0$ then the first equation $a^2=\sum _jd_jf_j^2$ implies that $a=f_1=\ldots =f_s=0$ as well, and hence $\eta =0$. Assume now $a\not= 0$. If $f_j\not=0$ then from the equation for $f_j$ we have
\begin{eqnarray*}
\frac{f_j}{a}=\frac{2d_j}{4d_j+2g_j-2}\in \mathbb{Q},
\end{eqnarray*}
and therefore
\begin{eqnarray*}
\eta =a(H-\sum _{j}\frac{f_j}{a}F_j)\in \mathbb{R}H^{1,1}(X,\mathbb{Q}),
\end{eqnarray*}
as wanted. 

c) Let $X\rightarrow Y$ be a finite composition of blowups along smooth centers, the images in $Y$ of whose exceptional divisors are points. We now show that if $Y$ satisfies the assumptions of Condition (A), then $X$ does also. Without loss of generality, we may assume that $X\rightarrow Y$ can be decomposed as $X=X_n\rightarrow X_{n-1}\rightarrow \ldots \rightarrow X_1\rightarrow Y$, where $X_1\rightarrow Y$ is a point blowup and the images of the exceptional divisors of $X\rightarrow Y$ is that point. 

We need to show that if $\eta \in \mathcal{B}(X)$ be such that $\eta .\eta =0$ then $\eta \in \mathbb{R}H^{1,1}(X,\mathbb{Q})$. We prove this by induction on $n$.

Initial case $n=1$: $X\rightarrow Y=\pi :X_1\rightarrow Y$ be blowup at one point. Let $E$ be the exceptional divisor and let $L$ be a line in $E$. Let $\eta \in \mathcal{B}(X)$ be so that $\eta .\eta =0$. Let $\xi =\pi _*(\eta )$. Then by Lemma \ref{LemmaPushforwardClassBBlowup}, $\xi \in \mathcal{B}(Y)$. We can write $\eta =\pi ^*(\xi )-\alpha E$ for some constant $\alpha\geq 0$. Computing as in Section 2 we obtain 
\begin{eqnarray*}
0=\eta .\eta =\pi ^*(\xi .\xi )-\alpha ^2L.
\end{eqnarray*}
Intersecting the RHS of the above equality with $E$, it follows that $\alpha =0$ and therefore $\pi ^*(\xi .\xi )=0$ as well. Hence $\xi .\xi =0$, and by the induction assumption, it follows that $\xi \in \mathbb{R}H^{1,1}(X,\mathbb{Q})$. Consequently, $\eta =\pi ^*(\xi )\in \mathbb{R}H^{1,1}(X,\mathbb{Q})$. 

Assume by induction that the claim is true for $n=j$. We prove that it is true for $n=j+1$. 

We consider two cases:

Case 1: $p:X_{j+1}\rightarrow X_j$ is a point blowup. Let $E=\mathbb{P}^2$ be the exceptional divisor and let $L\subset E$ be a line. Let $\eta \in \mathcal{B}(X_{j+1})$ be such that $\eta .\eta =0$. We need to show that $\eta \in \mathbb{R}H^{1,1}(X_{j+1},\mathbb{Q})$. Let us write $\xi =p_*(\eta )\in \mathcal{B}(X_j)$ and $\eta =p^*(\xi )-\alpha E$ for $\alpha \geq 0$. Then $\eta .\eta =0$ becomes $p^*(\xi .\xi )-\alpha ^2L=0$ and hence $p^*(\xi .\xi )=\alpha ^2L=0$ since they are linearly independent. Thus $\alpha =0$ and $\xi .\xi =0$. Apply induction assumption we get $\xi \in \mathbb{R}H^{1,1}(X_j,\mathbb{Q})$ and therefore $\eta =p^*(\xi )\in \mathbb{R}H^{1,1}(X_{j+1},\mathbb{Q})$. 

Case 2: $p:X_{j+1}\rightarrow X_j$ is a blowup at a smooth curve $C\subset X_j$ so that the push-forward of $C$ under the map $X_j\rightarrow Y$ is zero. Let $F$ be the exceptional divisor of the map $p$, and let $M\subset F$ be a fiber of the projection $F\rightarrow C$. Let $\eta \in \mathcal{B}(X_{j+1})$ be such that $\eta .\eta =0$. We need to show that $\eta \in \mathbb{R}H^{1,1}(X_{j+1},\mathbb{Q})$. Let us write $\xi =p_*(\eta )$ which is in $\in \mathcal{C}(X_j)$ by Lemma \ref{LemmaPushforwardClassBBlowup}, and $\eta =p^*(\xi )-\alpha F$ for $\alpha \geq 0$. Then $\eta .\eta =0$ becomes $p^*(\xi .\xi )-2\alpha p^*(\xi ).E+\alpha ^2E.E=0$. Push-forward this equation by $p$ we get $\xi .\xi =\alpha ^2C$. Apply Lemma \ref{LemmaSolveEquationEta2IsAnEffectiveDivisor}, it follows that $\xi .\xi =\alpha ^2C=0$. Hence $\alpha =0$, $\xi \in \mathcal{B}(X_j)$ and $\eta =p^*(\xi )$. Apply induction assumption for $\xi .\xi =0$ we get $\xi \in \mathbb{R}H^{1,1}(X_j,\mathbb{Q})$ and therefore $\eta =p^*(\xi )\in \mathbb{R}H^{1,1}(X_{j+1},\mathbb{Q})$. 

d) Let $\pi :X\rightarrow Y$ be the blowup of $Y$ along a smooth curve $C\subset Y$ so that $\gamma :=c_1(Y).C+2g-2<0$, and $C$ is not the only effective curve in its cohomology class. We now show that if $Y$ satisfies Condition (A), then $X$ does so. Let $F$ be the exceptional divisor of the blowup and let $M\subset F$ be a fiber of the projection $F\rightarrow C$.    

Let $\eta \in \mathcal{B}(X)$, then $\xi =\pi _*(\eta )\in \mathcal{B}(Y)$ by Lemma \ref{LemmaPushforwardClassBBlowup} and the assumption on $C$, and there is $a\geq 0$ so that $\eta =\pi ^*(\xi )-aF$. Assume that $\eta .\eta =0$. Then 
\begin{eqnarray*}
0&=&\eta .\eta .F=(\pi ^*(\xi .\xi )-2a\xi .F+a^2F.F).F=2a\xi .C-a^2\gamma . 
\end{eqnarray*} 
Here we used that $F.F.F=-\gamma $ and $\pi _*(F.F)=-C$. From this, it follows that $a=0$. Otherwise we can divide by $a>0$ and obtain $2\xi .C=a\gamma $ which is a contradiction since $\xi .C\geq 0$ (because $\xi \in \mathcal{B}(Y)$) and $\gamma <0$. Knowing $a=0$ we can argue as at the end of the proof of c).   

e) Let $\pi :X\rightarrow Y$ be the blowup of $Y$ along a smooth curve $C\subset Y$ so that there is an irreducible hypersurface $S\subset Y$ containing $C$ satisfying  condition 3) of Theorem \ref{TheoremAutomorphismBlowupP3}. As the last step of the proof of Theorem \ref{TheoremAutomorphismBlowupP3}, we now show that if $Y$ satisfies  Condition (A), then $X$ does so. Let $F$ be the exceptional divisor of the blowup and let $M\subset F$ be a fiber of the projection $F\rightarrow C$.

Let $\eta \in \mathcal{B}(X)$, then $\xi =\pi _*(\eta )\in \mathcal{C}(Y)$ by Lemma \ref{LemmaPushforwardClassBBlowup}, and there is $a\geq 0$ so that $\eta =\pi ^*(\xi )-aF$. Assume that $\eta .\eta =0$. Then 
\begin{eqnarray*}
0&=&\eta .\eta .F=(\pi ^*(\xi .\xi )-2a\xi .F+a^2F.F).F=2a\xi .C-a^2\gamma . 
\end{eqnarray*} 
Here we used that $F.F.F=-\gamma $ and $\pi _*(F.F)=-C$. From this, it follows that $a=0$. Otherwise we can divide by $a>0$ and obtain $2\xi .C=a\gamma $. We now construct an effective curve $C_0\subset F$ and use it to derive a contradiction.

Recall that $\kappa =S.C$, and $\mu \geq 1$ is the multiplicity of $C$ in $S$. Then the strict transform $\widetilde{S}$ of $S$ is given by $\widetilde{S}=\pi ^*(S)-\mu F$, and is an irreducible hypersurface of $X$. Since $\widetilde{S}$ and $F$ are different irreducible hypersurfaces, their intersection $C_0=\widetilde{S}.F=(\pi ^*(S)-\mu F).F$ is an effective curve of $F$. We now compute the numbers $C_0.C_0$ and $C_0.M$. We have
\begin{eqnarray*}  
C_0.C_0&=&\widetilde{S}|_{F}.\widetilde{S}|_{F}=\widetilde{S}.\widetilde{S}.F\\
&=&(\pi ^*(S)-\mu F).(\pi ^*(S)-\mu F).F=-2\mu\pi ^*(S).F.F+\mu ^2F.F.F\\
&=&2\mu S.C-\mu ^2\gamma =2\mu \kappa -\mu ^2\gamma .
\end{eqnarray*}
Denote by $\tau =C_0.C_0$ and $\mu _0=C_0.M$. Note that $\mu _0\not=0$, otherwise we have $C_0$ is a multiplicity of $M$, and hence $\pi _*(C_0)=0$. But from the definition of $C_0$ we can see that $\pi _*(C_0)=\mu C\not= 0$. Then by the computations at the end of Section 2, we have
\begin{eqnarray*}
F.F=-\frac{1}{\mu _0}C_0+\frac{1}{2}(\frac{\tau }{\mu _0^2}+\gamma )M.
\end{eqnarray*}
Pushforward this by the map $\pi$, using that $\pi _*(F.F)=-C$ and $\pi _*(C_0)=\mu C$ we have that $\mu _0=\mu$.  

By Lemma \ref{LemmaIntersectionOnF} and the above computation $\tau =2\mu \kappa -\mu ^2\gamma $, we obtain
\begin{eqnarray*}
F.C_0=\frac{1}{2}(\gamma \mu -\frac{\tau}{\mu})=\gamma \mu -\kappa .
\end{eqnarray*}

Because $\eta \in \mathcal{B}(X)$ and $2\xi .C=a\gamma $, it follows that 
\begin{eqnarray*}
0&\leq&\eta .C_0=(\pi ^*(\xi )-aF).C_0=\mu \xi .C-\frac{a}{2}(\gamma \mu -\frac{\tau}{\mu}),\\
&=&\frac{a}{2}\gamma \mu -\frac{a}{2}(\gamma \mu -\frac{\tau}{\mu})=\frac{a}{2}\frac{\tau}{ \mu} =a(\kappa -\frac{1}{2}\gamma \mu  ).
\end{eqnarray*}
This contradicts the assumptions that $2\kappa <\gamma \mu$ and $a>0$. Therefore $a=0$. Knowing that $a=0$, it follows that $\xi \in\mathcal{B}(Y)$ and we can use the induction assumption for it to have $\xi \in \mathbb{R}.H^{1,1}(Y,\mathbb{Q})$ and therefore $\eta =\pi ^*(\xi )\in \mathbb{R}.H^{1,1}(X,\mathbb{Q})$.
\end{proof}

\begin{proof} (Of Theorem \ref{TheoremDynamicalDegreesAutomorphismsBlowupP3})

Let $\pi :X=X_n\rightarrow X_{n-1}\rightarrow \ldots \rightarrow X_1\rightarrow X_0=\mathbb{P}^3$ be a finite composition of blowups along smooth centers satisfying conditions of Theorem \ref{TheoremDynamicalDegreesAutomorphismsBlowupP3}. 

We first show the following: 

1) Claim 1: If $\eta \in \mathcal{C}(X)$ and $\eta \not= 0$, then either $\eta .\eta \not= 0$ or $K_X.\eta \not= 0$. Here $K_X$ is the canonical divisor of $X$. 

Proof (of Claim 1): We prove the claim by induction on $n$.

The initial case $n=0$: Then $X=\mathbb{P}^3$, $H^{1,1}(X)$ is generated by a generic hyperplane $H\subset \mathbb{P}^3$, $K_X=-4H$ and $\eta =aH$ for some $a>0$. Hence both $\eta .\eta $ and $K_X.\eta $ are non-zero.

Assume that the claim is true for $n=j$. We now show that it is true for $n=j+1$. We define by $p$ the map $X_{j+1}\rightarrow X_j$. Let $\eta \in \mathcal{C}(X_{j+1})$, which is non-zero, we now show that at least one of the expressions $\eta .\eta $ and $K_{X_{j+1}}.\eta \not=0$. 

We define $\xi =p_*(\eta )\in \mathcal{C}(X_j)$. We consider two cases:

Case 1: $p:X=X_{j+1}\rightarrow X_j=Y$ is a point blowup. Let $E=\mathbb{P}^2$ be the exceptional divisor of $p$ and let $L\subset E$ be a line. Then $K_X=p^*(K_Y)+2E$. We can write $\eta =p^*(\xi )-\alpha E$ for $\alpha =\eta .L\geq 0$.   

If we had both $\eta .\eta =0$ and $K_X.\eta =0$ then we have 
\begin{eqnarray*}
0&=&\eta .\eta =(p^*(\xi )-\alpha E).(p^*(\xi )-\alpha E)=p^*(\xi .\xi )-\alpha ^2L,\\
0&=&K_X.\eta =(p^*(K_Y )+2E).(p^*(\xi )-\alpha E)=p^*(K_Y .\xi )+2\alpha L.
\end{eqnarray*}
Since $p^*(\xi .\xi )$ and $L$ are linearly independent, from the first equation we imply that $\xi .\xi =0$ and $\alpha =0$. Similarly, from the second equation we have $K_Y.\xi =0$. If $\xi \not= 0$, then by induction assumption, not both $\xi .\xi $ and $K_Y .\xi $ are zero, and we arrive at a contradiction. If $\xi =0$, then $\eta =p^*(\xi )=0$ as well, and we have a contradiction again. Therefore Claim 1 is proved in Case 1.

Case 2: $p:X=X_{j+1}\rightarrow X_j=Y$ is a blowup of $Y$ along a smooth curve $C\subset Y$ for which $c_1(Y).C\not= 2g-2$, where $c_1(Y)=-K_Y$ is the first Chern class of $Y$ and $g$ is the genus of $C$. Let $F$ be the exceptional divisor of $p$, and let $M\subset F$ be a fiber of the projection $F\rightarrow C$. Let $g$ be the genus of $C$, and let $c_1(Y)=-K_Y$ be the first Chern class of $Y$. Then $K_X=p^*(K_Y)+F$. We can write $\eta =p^*(\xi )-\alpha F$ for $\alpha =\eta .L\geq 0$. 

We define $$\gamma =c_1(Y).C+2g-2.$$ If we had both $\eta .\eta =0$ and $K_X.\eta =0$ then we have 
\begin{eqnarray*}
0&=&\eta .\eta =(p^*(\xi )-\alpha F).(p^*(\xi )-\alpha F)=p^*(\xi .\xi )-2\alpha p^*(\xi ).F+\alpha ^2F.F,\\
0&=&K_X.\eta =(p^*(K_Y )+F).(p^*(\xi )-\alpha F)=p^*(K_Y .\xi )-\alpha p^*(K_Y).F+p^*(\xi ).F-\alpha F.F.
\end{eqnarray*}
Intersecting both of these equations with $F$, using $F.F.F=-\gamma$ and $p_*(F.F)=-C$, we obtain
\begin{eqnarray*}
&&2\alpha \xi .C-\alpha ^2\gamma =0,\\
&&\alpha K_Y.C-\xi .C+\alpha \gamma =0.
\end{eqnarray*}

Then we must have $\alpha =0$. Otherwise, dividing $2\alpha$ from the first equation we have that $\xi .C=\alpha \gamma /2$. Substituting this into the second equation and dividing by $\alpha$ we get $2K_Y.C=-\gamma$. Hence $c_1(Y).C=2g-2$, which is a contradiction.

Now that we have $\alpha =0$, the original equations become $p^*(\xi .\xi )=0$ and $p^*(K_Y.\xi )+p^*(\xi ).E=0$. Push-forward both of these equations to $Y$, we obtain that $\xi .\xi =0$ and $K_Y.\xi =0$ and can proceed as in Case 1.

2) Now we continue with the proof of Theorem \ref{TheoremDynamicalDegreesAutomorphismsBlowupP3}. Let $f\in Aut(X)$.

We first show that $\lambda _2(f)\geq \lambda _1(f)$. To this end, let $\eta$ be a non-zero nef class which is an eigenvector of eigenvalue $\lambda _1(f)\geq 1$ of $f^*:H^{1,1}(X)\rightarrow H^{1,1}(X)$. If $\eta .\eta \not= 0$, then $\eta .\eta$ is an eigenvector of eigenvalue $\lambda _1(f)^2$ of $f^*:H^{2,2}(X)\rightarrow H^{2,2}(X)$, therefore $\lambda _2(f)\geq \lambda _1(f)^2\geq \lambda _1(f)$ as claimed. Otherwise, by Claim 1 we must have $K_X.\eta \not= 0$. Since $f$ is an automorphism of $X$, $f^*(K_X)=K_X$. Therefore $K_X.\eta $ is an eigenvector of eigenvalue $\lambda _1(f)$ of $f^*:H^{2,2}(X)\rightarrow H^{2,2}(X)$, and we again have $\lambda _2(f)\geq \lambda _1(f)$.  
 
If we apply the above argument to $f^{-1}$, we obtain $\lambda _1(f)=\lambda _2(f^{-1})\geq \lambda _1(f^{-1})=\lambda _2(f)$. Therefore $\lambda _1(f)=\lambda _2(f)$, and we are done.
\end{proof}

\section{Examples}
 
\subsection{The case $X_0=\mathbb{P}^2\times \mathbb{P}^1$}
By K\"unneth's formula, $H^{1,1}(X_0)$ is generated by the classes of $\mathbb{P}^2\times \{pt\}$ and $\mathbb{P}^1\times \mathbb{P}^1$ (here $\{pt\}$ means a point). The intersection on $H^{1,1}(X_0)$ is 
\begin{eqnarray*}
\mathbb{P}^2\times \{pt\}.\mathbb{P}^2\times \{pt\}&=&0,\\
\mathbb{P}^2\times \{pt\}.\mathbb{P}^1\times \mathbb{P}^1&=&\mathbb{P}^1\times \{pt\},\\
\mathbb{P}^1\times \mathbb{P}^1.\mathbb{P}^1\times \mathbb{P}^1&=&\{pt\}\times \mathbb{P}^1.
\end{eqnarray*}
By K\"unneth's formula again, $H^{2,2}(X_0)$ is generated by $\mathbb{P}^1\times \{pt\}$ and $\{pt\}\times \mathbb{P}^1$. The pairing between $H^{1,1}(X_0)$ and $H^{2,2}(X_0)$ is given by
\begin{eqnarray*}
\mathbb{P}^2\times \{pt\}.\mathbb{P}^1\times \{pt\}&=&0,\\
\mathbb{P}^2\times \{pt\}.\{pt\}\times \mathbb{P}^1&=&1,\\
\mathbb{P}^1\times \mathbb{P}^1.\mathbb{P}^1\times \{pt\}&=&1,\\
\mathbb{P}^1\times \mathbb{P}^1.\{pt\}\times \mathbb{P}^1&=&0.
\end{eqnarray*}

a) We first check that the space $X_0=$ satisfies Condition (A) in the proof of Theorem \ref{TheoremAutomorphismBlowupP3}. Let $\eta $ be in $H^{1,1}_{nef}(X_0)$ so that $\eta .\eta =0$. We need to show that $\eta \in \mathbb{R}.H^{1,1}(X_0,\mathbb{Q})$. In fact, we can write $\eta =a \mathbb{P}^2\times \{pt\}+b\mathbb{P}^1\times \mathbb{P}^1$ for real numbers $a$ and $b$. Since $\eta $ is nef, we have
\begin{eqnarray*}
 b=\eta .\mathbb{P}^1\times \{pt\}&\geq&0,\\
 a=\eta .\{pt\}\times \mathbb{P}^1&\geq&0.
\end{eqnarray*}
By computation, it follows that $\eta .\eta =2ab\mathbb{P}^1\times \{pt\}+b^2\{pt\}\times \mathbb{P}^1$. Therefore $\eta .\eta =0$ if and only if $ab=b^2=0$, i.e. $b=0$. Hence $\eta =a\mathbb{P}^2\times \{pt\}\in \mathbb{R}.H^{1,1}(X_0,\mathbb{Q})$, as wanted. 

b) We next show the following: Let $p_1,\ldots ,p_n$ and $p$ be pairwise distinct points in $\mathbb{P}^2$. Let $\pi :X\rightarrow X_0=\mathbb{P}^2\times \mathbb{P}^1$ be the blowup at curves $p_1\times \mathbb{P}^1,\ldots ,p_n\times \mathbb{P}^1$, and let $C=p\times \mathbb{P}^1$. Then $C$ does not satisfy both conditions 2) and 3) in Theorem \ref{TheoremAutomorphismBlowupP3}. Moreover, note that some of these spaces $X$ does have automorphisms of positive entropies (see Remarks right after Theorem \ref{TheoremAutomorphismBlowupP3}). Therefore, we see that the conditions of Theorem \ref{TheoremAutomorphismBlowupP3} are somewhat optimal. 

\begin{proof}
Using Whitney's sum formula for Chern classes, we find that $c_1(X_0)=2\mathbb{P}^2\times \{pt\}+3\mathbb{P}^1\times \mathbb{P}^1$. Denote by $Z$ the blowup of $\mathbb{P}^2$ at the points $p_1,\ldots ,p_n$, and let $E_1,\ldots ,E_n$ be the corresponding exceptional divisors. Then $X$ is biholomorphic equivalent to $Z\times \mathbb{P}^1$, and $c_1(X)=\pi ^*(2\mathbb{P}^2\times \{pt\}+3\mathbb{P}^1\times \mathbb{P}^1)-\sum _jE_j\times \mathbb{P}^1$. The curve $C=p\times \mathbb{P}^1$ has genus $g=0$, and has intersections $0$ with the exceptional divisors $E_j\times \mathbb{P}^1$ since $p$ is different from $p_j$. Therefore
\begin{eqnarray*}
\gamma =c_1(X).C+2g-2&=&(2\mathbb{P}^2\times \{pt\}+3\mathbb{P}^1\times \mathbb{P}^1).p\times \mathbb{P}^1-2\\
&=&2-2=0.  
\end{eqnarray*}
Thus $C$ does not satisfy condition 2) of Theorem \ref{TheoremAutomorphismBlowupP3}.

Now let $\widetilde{S}$ be an irreducible hypersurface of $X$ containing $C$. Then $\widetilde{S}$ can not be one of the exceptional divisors $E_j\times \mathbb{P}^1$, since $p\not= p_j$. Therefore, $\widetilde{S}$ must be the strict transform of a hypersurface $S\subset X_0$. Therefore, in cohomology: $\widetilde{S}=\pi ^*(S)-\sum _{j}\mu _jE_j\times \mathbb{P}^1$, here $\mu _j$ is the multiplicity of $p_j\times \mathbb{P}^1$ in $S$. If we let $p$ vary, we see that the curve $C$ moves in a family of curves which cover the whole space $X_0=\mathbb{P}^2\times \mathbb{P}^1$. Hence there must be a curve in the family intersecting $S$ at isolated points, and this shows that $S.C\geq 0$. Whatever the multiplicity $\mu$ of $C$ in $\widetilde{S}$ is, then $\mu .\gamma =0$. Also, $C$ has intersections $0$ with exceptional divisors $E_j\times \mathbb{P}^1$ as above. Therefore 
\begin{eqnarray*}  
\kappa =\widetilde{S}.C=S.C\geq 0=\mu \gamma .
\end{eqnarray*}
Thus condition 3) in Theorem \ref{TheoremAutomorphismBlowupP3} is not satisfied for $C$.
\end{proof}

From the above proof we obtain the following consequence
\begin{corollary}
Let $\pi :S\rightarrow \mathbb{P}^2$ be a finite blowup of $\mathbb{P}^2$. Then for any automorphism $f:S\times \mathbb{P}^1\rightarrow S\times \mathbb{P}^1$ we have $\lambda _1(f)=\lambda _2(f)$. 
\label{CorollaryDynamicalDegreeBlowupP2P1}\end{corollary}
\begin{proof}
We can see from the computation in the above proof that Theorem \ref{TheoremDynamicalDegreesAutomorphismsBlowupP3} applies: the manifold $X=S\times \mathbb{P}^1$ is obtained as a finite composition of blowups  $X_{j+1}\rightarrow X_{j}$ along curves $C_j$ which are isomorphic to $\mathbb{P}^1$, hence $c_1(X_j).C_j=2\not= 2g-2 =-2$. 
\end{proof}
We observe that Corollary \ref{CorollaryDynamicalDegreeBlowupP2P1} is compatible with the fact that known examples of automorphisms of positive entropies of $S\times \mathbb{P}^1$ are products $g\times h:S\times \mathbb{P}^1\rightarrow S\times \mathbb{P}^1$. 

c) Finally we show the following: Let $p_1,\ldots ,p_n$ and $p$ be pairwise distinct points in $\mathbb{P}^1$. Let $D_1,\ldots ,D_n$ and $D$ be smooth curves in $\mathbb{P}^2$. Let $\pi :X\rightarrow X_0=\mathbb{P}^2\times \mathbb{P}^1$ be the blowup at curves $D_1\times p_1,\ldots ,D_n\times p_n$. Then the curve $C=D\times p$ does not satisfy condition 2) of Theorem \ref{TheoremAutomorphismBlowupP3}, but it does satisfy condition 3) of Theorem \ref{TheoremAutomorphismBlowupP3}. Therefore, any automorphism of $X$ has topological entropy zero. 

\begin{proof}
Let $E_j$ be the exceptional divisor of the blowup corresponding to $D_j\times p_j$. Then 
\begin{eqnarray*}
c_1(X)= \pi ^*(2\mathbb{P}^2\times \{pt\}+3\mathbb{P}^1\times \mathbb{P}^1)-\sum _jE_j.
\end{eqnarray*}
Let $d\geq 1$ be the degree of $D$, and let $g\geq 0$ be its genus. Note that $C=D\times p$ is disjoint from the curves $D_j\times p_j$ since $p\not= p_j$, hence $C$ has intersection $0$ with the exceptional divisors $E_j$. Therefore
\begin{eqnarray*}
\gamma =c_1(X).C+2g-2&=&(2\mathbb{P}^2\times \{pt\}+3\mathbb{P}^1\times \mathbb{P}^1).D\times p=3d+2g-2\geq 1.
\end{eqnarray*}  
This shows that $C$ does not satisfy condition 2) in Theorem \ref{TheoremAutomorphismBlowupP3}.

Let $S=\mathbb{P}^2\times p\subset \mathbb{P}^2\times \mathbb{P}^1$, which can be  identified with its strict transform in $X$ since $S$ has no intersection with the centers of blowups. Then $S$ is an irreducible hypersurface in $X$ containing $C=D\times p$, and the multiplicity of $C$ in $S$ is $\mu =1$. Moreover, 
\begin{eqnarray*}
\kappa =S.C=\mathbb{P}^2\times p.D\times p=0.    
\end{eqnarray*}
Hence $2\kappa =0<1\leq \mu \gamma $, which shows that condition 3) of Theorem \ref{TheoremAutomorphismBlowupP3} is satisfied. 
\end{proof}

\subsection{The case $X_0=\mathbb{P}^1\times \mathbb{P}^1\times \mathbb{P}^1$} This case is very similar to the case $X_0=\mathbb{P}^2\times \mathbb{P}^1$ above. The readers can easily redo all the (analogs of) computations and constructions in the previous section. 
 
\subsection{Proofs of Examples 3, 4, 5 and 6.}
\begin{proof} (Of Example 3)
Let $E_1,\ldots ,E_t$ be the exceptional divisors of the blowup $Y\rightarrow \mathbb{P}^3$, and let $L_1\subset E_1,\ldots ,L_t\subset E_t$ be lines. Let $H$ be a generic hyperplane in $\mathbb{P}^3$. Let $C$ be the strict transform of $D$ and $S$ is the strict transform of $W$. Then their classes are
\begin{eqnarray*}
C&=&dH.H-\sum _iL_i,\\
S&=&H-\sum _iE_i,
\end{eqnarray*}
while $c_1(Y)=4H-2\sum _iE_i$. Since $D$ is a smooth plane curve, by the genus formula, the genus $g$ of $D$ is $g=(d-1)(d-2)/2$ which is the same as that of $C$. Therefore,
\begin{eqnarray*}
\kappa &=&S.C=d-t,\\
\gamma &=&c_1(Y).C+2g-2=4d-2t+(d-1)(d-2)-2,\\
\mu &=&1. 
\end{eqnarray*}
Hence the inequality $2\kappa <\mu \gamma$ is the same as
\begin{eqnarray*}
2d+(d-1)(d-2)-2>0,
\end{eqnarray*}
which is satisfied when $d\geq 2$. For the case $d=1$, the proof that $X$ satisfies Theorem \ref{TheoremDynamicalDegreesAutomorphismsBlowupP3} when $t\not= 3$ can be done similarly to the above, see Example 5 for an explicit calculation when $t=2$.  
\end{proof}

\begin{proof} (Of Example 4)
The blowup $Y\rightarrow \mathbb{P}^3$ is the blowup $X_1\rightarrow \mathbb{P}^3$ of Theorem \ref{TheoremAutomorphismBlowupP3}, therefore satisfies Condition (A) in the proof of Theorem \ref{TheoremAutomorphismBlowupP3}. Let $F$ be the exceptional divisor of the blowup $Y\rightarrow \mathbb{P}^3$, and let $M$ be a fiber of $F\rightarrow C_1$.

Let $d_1$ and $d_2$ be the degrees of $C_1$ and $C_2$. Then $C_1$ and $C_2$ intersect at $d_1.d_2$ points in $W$, by Bezout theorem. Let $C$ be the strict transform of $C_2$ in $Y$, then its class is $C=d_2H.H-d_1d_2M$. Let $S$ be the strict transform of $W$ in $Y$. Then $S$ contains $C$ and its class is $S=H-F$. The first Chern class of $Y$ is $c_1(Y)=4H-F$. 

We now check that $C$ satisfies condition 3) of Theorem \ref{TheoremAutomorphismBlowupP3}. We have $\mu =1$, 
\begin{eqnarray*}
\gamma =c_1(Y).C+2g-2=4d_2-d_1d_2+2g-2,
\end{eqnarray*}
and
\begin{eqnarray*}
\kappa =S.C=d_2-d_1d_2.
\end{eqnarray*}
Therefore, since $d_1,d_2\geq 1$ and $g\geq 0$, we have
\begin{eqnarray*}
\mu \gamma -2\kappa =2d_2+d_1d_2+2g-2 >0,
\end{eqnarray*}
as wanted. 
\end{proof}

\begin{proof} (Of Example 5)
Let $E_0$, $E_1$, $E_2$ and $E_3$ be the exceptional divisors of the blowup $Y\rightarrow \mathbb{P}^3$, and let $L_0\subset E_0$,  $L_1\subset E_1$, $L_2\subset E_2$ and $L_3\subset E_3$ be the generic lines. Let $H\subset \mathbb{P}^3$ be a generic hyperplane. Then $C=\widetilde{\Sigma _{0,1}}=H.H-L_2-L_3$, degree of $C$ is $d=1$ and its genus is $g=0$. The first Chern class of $Y$ is 
\begin{eqnarray*}
c_1(Y)=4H-2E_0-2E_1-2E_2-2E_3.
\end{eqnarray*}
Therefore 
\begin{eqnarray*}
c_1(Y).C+2g-2=4-2-2+0-2=-2<0.
\end{eqnarray*}
Similarly we can check for other curves $\widetilde{\Sigma _{i,j}}$. However, these curves are the unique effective curves in their cohomology classes, thus Theorem \ref{TheoremAutomorphismBlowupP3} does not apply. Theorem \ref{TheoremDynamicalDegreesAutomorphismsBlowupP3} does apply though, since $0=c_1(Y).C\not= 2g-2=-2$.   
\end{proof}

\begin{proof} (Of Example 6)
Let $p:Y\rightarrow \mathbb{P}^3$ be the blowup of $\mathbb{P}^3$ at $e_1$ and $e_3$. Let $E_1$ and $E_3$ be the exceptional divisors, and let $L_1\subset E_1$ and $L_2\subset E_2$ be generic lines. Let $H\subset \mathbb{P}^3$ be a generic hyperplane. Since $\Sigma _{0,1}$ contains $e_3$, the class of its strict transform $\widetilde{\Sigma _{0,1}}$ in $Y$ is $H.H-L_3$. Because 
$$c_1(Y)=p ^*(c_1(\mathbb{P}^3))-2E_1-2E_3=4H-2E_2-2E_3,$$ 
and the genus of $\widetilde{\Sigma _{0,1}}$ is $g=0$, we have
\begin{eqnarray*} 
c_1(Y).\widetilde{\Sigma _{0,1}}+2g-2=(4H-2E_2-2E_3)(H.H-L_3)-2=4-2-2=0.
\end{eqnarray*}
Therefore $\widetilde{\Sigma _{0,1}}$ does not satisfy condition 2) in Theorem \ref{TheoremAutomorphismBlowupP3}, and it does not satisfy conditions 1) and 3) as well. But we can compute directly (as in part b) of the proof of Theorem \ref{TheoremAutomorphismBlowupP3}) to show that the space $Z$, which is the blowup of $Y$ at $\widetilde{\Sigma _{0,1}}$ satisfies Condition (A) in the proof of Theorem \ref{TheoremAutomorphismBlowupP3}. Let $F$ be the exceptional divisor of the blowup $q:Z\rightarrow Y$ and let $M\subset F$ be a fiber of $F\rightarrow \widetilde{\Sigma _{0,1}}$. Since $e_1\in \Sigma _{0,3}$ and $\Sigma _{0,3}\cap \Sigma _{0,1}=e_2\not= e_1,e_3$, the class of the strict transform $\widetilde{\Sigma _{0,3}}$ in $Z$ is $H.H-L_1-M$. Meanwhile
\begin{eqnarray*}
c_1(Z)=q^*(c_1(Y))-F=4H-2E_1-2E_3-F. 
\end{eqnarray*}
Therefore, since the genus of $\widetilde{\Sigma _{0,3}}$ is $g=0$, it follows that $c_1(Z).\widetilde{\Sigma _{0,3}}+2g-2=-1<0$. Moreover, $\widetilde{\Sigma _{0,3}}$ is not the only effective curve in its cohomology class (its cohomology class is the same as the class of the strict transform of a generic line passing to $e_3$ and intersecting $\Sigma _{0,1}$).  Part d) of the proof of Theorem \ref{TheoremAutomorphismBlowupP3} implies that $X$ satisfies Condition (A). Therefore any automorphism of $X$ has zero entropy.

Alternatively, we can show that $\widetilde{\Sigma _{0,3}}$ satisfies condition 3) of Theorem \ref{TheoremAutomorphismBlowupP3}. We let $S$ be the strict transform of the hyperplane $\Sigma _0=\{x_0=0\}\subset \mathbb{P}^3$. Then $S$ contains $\widetilde{\Sigma _{0,3}}$ with multiplicity $\mu =1$. The class of $S$ is $S=H-E_1-E_3-F$. Therefore $\kappa =S.\widetilde{\Sigma _{0,3}}=-1$, and 
\begin{eqnarray*}
2\kappa =-2<-1=c_1(Z).\widetilde{\Sigma _{0,3}}+2g-2=\mu \gamma .
\end{eqnarray*}

\end{proof}

\end{document}